\documentclass[a4paper,11pt, reqno]{amsart}
\usepackage{amssymb,amsthm,amsmath}
\usepackage{amsmath}
 \usepackage{enumerate}
\usepackage{ifthen}
\usepackage{graphicx}
\usepackage{cite}
\usepackage{tikz}
\usepackage{comment}
\usepackage{amsmath,amscd,amssymb}
\usepackage{latexsym}
\usepackage[colorlinks,citecolor=blue,pagebackref,hypertexnames=false]{hyperref}

\numberwithin{equation}{section}
\allowdisplaybreaks[2]
\theoremstyle{plain}
\newtheorem{theorem}{Theorem}[section]

\newtheorem{lemma}[theorem]{Lemma}

\theoremstyle{definition}
\newtheorem{definition}[theorem]{Definition}

\theoremstyle{remark}
\newtheorem{remark}[theorem]{Remark}

\newtheorem{case[theorem]}{Case}

\def\norm#1.#2.{\lVert#1\rVert_{#2}}

\title[Dispersive decay for the mass-critical Sch\"odinger equation when $d\geq 3$]{Dispersive decay for the mass-critical Sch\"odinger equation when $d\geq 3$}

\author{Jiabin Qian}
\author{Manli Song}
\thanks{Corresponding author: Manli Song}

\address{\endgraf School of Mathematics and Statistics, Northwestern Polytechnical University, Xi'an, Shaanxi 710129, China}\email{qjb330424@mail.nwpu.edu.cn}

\address{\endgraf School of Mathematics and Statistics, Northwestern Polytechnical University, Xi'an, Shaanxi 710129, China} \email{mlsong@nwpu.edu.cn}

\date{\today}
\begin{document}

\begin{abstract}
In this paper we establish the pointwise-in-time dispersive decay for solutions to the mass-critical nonlinear Schr\"odinger equation in spatial dimensions $d\geq3$.  Our argument relies on a delicate decomposition of the nonlinearity and an improved linear estimate, which together enable us to control the nonlinear contribution. This work unifies a framework for
extending the foundational results established in an earlier paper by Fan, Killip, Visan, and Zhao [Math. Z. \textbf{311}(1), Paper No. 21, 16 pp (2025)], where the same problem was addressed for spatial dimensions $d=1,2,3$, to the general higher-dimensional setting $d\geq3$.
\end{abstract}

	\maketitle

	\allowdisplaybreaks

	\tableofcontents

	\section{Introduction}
The paper studies the long-time behavior of solutions to the following mass-critical nonlinear Schr\"odinger equation:
\begin{equation}\label{NLS}
\begin{cases}
&i u_{t} +\Delta u = F(u), \\
&u(0,x) = u_0(x)\in L^2(\mathbb{R}^d),
\end{cases}
\end{equation}
where $F(u)=\pm|u|^\frac{4}{d} u$ and $u(t,x)$ is a complex-valued function on spacetime $\mathbb{R}\times\mathbb{R}^d$. With the plus sign, the equation \eqref{NLS} is said to be defocusing; with the minus sign, it is said to be focusing. This equation is mass-critical, which refers to  the scaling symmetry associated with \eqref{NLS}, namely
\begin{equation}
u(t,x)\mapsto u_\lambda(t,x):=\lambda^\frac{d}{2}u(\lambda^2 t,\lambda x), \text{ for }\lambda>0,
\end{equation}
and also preserves the (conserved) mass of the solution
\begin{equation}\label{mass-preserved}
M(u(t))=\int_{\mathbb{R}^d}|u(t,x)|^2dx=M(u_0).
\end{equation}

In a seminal series of contributions \cite{Dodson2012, Dodson2015, Dodson2016, Dodson}, Dodson established the global well-posedness and scattering theory for solutions to \eqref{NLS}. In the focusing case, global wellposedness holds under the additional assumption that the mass of the initial data is smaller than that of the ground state $Q$, which is the unique
non-negative, radial, $H^1(\mathbb{R}^d)$-solution to the elliptic equation
\begin{equation*}
\Delta Q+|Q|^\frac{4}{d} Q=Q.
\end{equation*}
This mass restriction is sharp: finite time blowup can occur for solutions to (1.1) with initial data with mass $M(u_0)>M(Q)$. In the following theorem, we summarize the well-posedness results that are
needed.
\begin{theorem}[\cite{Dodson2012, Dodson2015, Dodson2016, Dodson}]\label{Dodson} Fix $d\geq 1$ and let $u_0\in L^2(\mathbb{R}^d)$. In the focusing case assume also that $M(u_0)<M(Q)$. Then there exists a unique global solution $u\in C_tL_x^2\cap L_{t,x}^\frac{2(d+2)}{d}(\mathbb{R}\times\mathbb{R}^d)$ to \eqref{NLS} with initial data
$u_0$ which satisfies
\begin{equation}\label{p=q decay}
\|u\|_{L_{t,x}^\frac{2(d+2)}{d}(\mathbb{R}\times\mathbb{R}^d)}\leq C(\|u_0\|_{L_x^2}).
\end{equation}
Moreover, there exist asymptotic states $u_\pm\in L^2(\mathbb{R}^d)$ such that
\begin{equation}\label{scattering}
\lim\limits_{t\rightarrow \pm\infty}\|u(t)-e^{it\Delta}u_\pm(x)\|_{L^2_x}=0.
\end{equation}
\end{theorem}

The assertion \eqref{scattering} is commonly referred to as \textit{scattering}, which demonstrates that solutions to the nonlinear equation \eqref{NLS} bear an asymptotic resemblance to those of the linear Schrödinger equation. Notably, the space-time bound \eqref{Global} also applies to the linear Schr\"odinger equation, and this is precisely a prototypical \textit{Strichartz inequality}. Then it naturally raises the inquiry of which properties of the linear Schr\"odinger are shared by solutions to the nonlinear equation \eqref{NLS}. Specifically, one salient characteristic exhibited by the linear Schrödinger equation is \textit{dispersive decay}: for any $2\leq r\leq\infty$,
\begin{equation}\label{classical decay}
\|e^{it\Delta}u_0\|_{L_x^r}\lesssim |t|^{-d(\frac{1}{2}-\frac{1}{r})}\|u_0\|_{L^{r'}_x},\;\forall t\neq0.
\end{equation}
Indeed, it is obtained via interpolation between the conservation of mass
\begin{equation*}
\|e^{it\Delta}u_0\|_{L_x^2}=\|u_0\|_{L^2_x},
\end{equation*}
which follows from the Plancherel theorem, and the dispersive estimate
\begin{equation*}
\|e^{it\Delta}u_0\|_{L_x^\infty}\lesssim|t|^{-\frac{d}{2}}\|u_0\|_{L^1_x},
\end{equation*}
which arises from the fundamental solution to the linear Schr\"odinger equation. Before the advent of Strichartz estimates, estimate of the form \eqref{classical decay}, namely, quantitative decay pointwise in time, was the key tool to understand the long-time behavior of solutions. For this reason, the dispersive decay estimate was closely related to the study of well-posedness and scattering and required smooth initial data. See \cite{Hayashi-Tsutsumi, Klainerman1985, Klainerman-Ponce, Lin-Strauss, Morawetz-Strauss, Segal} as well as the references therein.

Motivated by the scattering property \eqref{scattering} and the linear dispersive decay \eqref{classical decay}, it is pertinent to investigate whether solutions to the scaling-critical nonlinear Schr\"odinger equation likewise possess the characteristic of dispersive decay. Ma, Wang, Yu and Zhao \cite{Ma-Wang-Yu-Zhao} established a dispersive decay estimate for initial data with $H^{1+}$ regularity. Their work primarily focused on the hyperbolic space setting, with a subsequent discussion of extension to $\mathbb{R}^3$. This result improved earlier findings requiring initial data with $H^3$ regularity, see \cite{Fan-Zhao, Guo-Huang-Song}. Subsequent work has significantly lowered the regularity required for dispersive decay. For the defocusing cubic nonlinear Schr\"odinger equation on $\mathbb{R}^3$, Fan, Staffilani and Zhao \cite{Fan-Staffilani-Zhao} showed that the decay rate of nonlinear solutions is comparable to that of linear solutions, with initial data regularity substantially weaker than in previous studies, see \cite{Fan-Zhao, Fan-Zhao2023}. However, prior to \cite{Fan-Killip-Visan-Zhao}, these results all required strictly higher regularity than that needed for well-posedness or scattering, and failed to establish a linear dependence on the initial data in the sense of \eqref{classical decay}. In \cite{Fan-Killip-Visan-Zhao}, Fan, Killip, Visan, and Zhao demonstrated dispersive decay for the solution to the mass-critical nonlinear Schr\"odinger equation \eqref{NLS} with initial data belonging to $L^2$ in spatial dimensions $d=1,2,3$. A key insight is that this result only requires the initial data to lie in the scaling-critical space $L^2$. For the spatial dimensions $d=1,2$, the proof was established by parallel methods, using crucially Lorentz-space improvements of the traditionally Strichartz inequality; For $d=3$, however, this approach is inapplicable due to the low power of nonlinearity. Instead, the result was derived through a delicate decomposition of the nonlinearity and the unusual linear estimates. With Leveraging the approach in \cite{Fan-Killip-Visan-Zhao}, Kowalski \cite{Kowalski} reduced the required regularity to the scaling-critical spaces $\dot{H}^1$ for energy-critical nonlinear Schr\"odinger equation in spatial dimensions $d=3,4$.

In this paper, we prove the dispersive decay, pointwise in time, for solutions to the mass-critical nonlinear Schr\"odinger equation \eqref{NLS} in higher spatical dimensions $d\geq3$.
\begin{theorem}\label{main result} Fix $d\geq3$ and $2\vee\frac{2d(d-1)}{d^2-2d+4}<r<\frac{2d}{d-2}$. Given any $u_0\in L^2\cap L^{r'}(\mathbb{R}^d)$ satisfying the hypotheses of Theorem \ref{Dodson}, let $u$ denote the unique global solution to \eqref{NLS} with initial data $u_0$. Then
\begin{equation}\label{main}
\sup\limits_{t\neq 0}|t|^{d(\frac{1}{2}-\frac{1}{r})}\|u(t)\|_{L_x^r}\leq C(\|u_0\|_{L_x^2})\|u_0\|_{L_x^{r'}}.
\end{equation}
\end{theorem}
\begin{remark}The key novelty of this paper can be seen as follows:
\begin{itemize}
\item For general higher dimensions $d\geq 3$, complications arise due to the power of the nonlinearity. To address the issue, we carry out a delicate decomposition of it by adapting the method in \cite{Fan-Killip-Visan-Zhao}.
Nevertheless, the unusual linear estimates in \cite[Lemma 2.6]{Fan-Killip-Visan-Zhao} are inadequate to yield the dispersive decay. Accordingly, we derive an improved linear estimate in Lemma \ref{key lemma} which can be found in Section \ref{section 2}.
\item In the proof of Theorem \ref{main result}, a critical step involves partitioning $r$ into three refined cases: $\frac{2d(d-1)}{d^2-2d+4}\vee 2<r<\frac{2(d+1)}{d}$, $\frac{2(d+1)}{d}\leq r\leq \frac{2(d+2)}{d}$ and $\frac{2(d+2)}{d}<r<\frac{2d}{d-2}$. Notably, for $\frac{2d(d-1)}{d^2-2d+4}\vee 2<r<\frac{2(d+1)}{d}$, the inequality $\frac{qr}{qr-2}\cdot\frac{d+4}{d}>r$ holds for any Schr\"odinger-admissible pair $(p,q)$. When $\frac{2(d+2)}{d}<r<\frac{2d}{d-2}$, the inverse inequality $\frac{qr}{qr-2}\cdot\frac{d+4}{d}<r$ holds for any Schr\"odinger-admissible pair $(p,q)$. In the intermediate range $\frac{2(d+1)}{d}\leq r\leq \frac{2(d+2)}{d}$, there exists a unique Schr\"odinger-admissible pair $(p,q)$ such that $\frac{qr}{qr-2}\cdot\frac{d+4}{d}=r$.
\item We develop a unified framework to establish the dispersive decay for the mass-critical Schr\"odinger equation when $d\geq3$. A key feature here is that no auxiliary assumptions are imposed beyond the initial data belonging to the scaling-critical space $L^2$, which is a minimal requirement for the existence of global solutions. It underlies the limitation $r<\frac{2d}{d-2}$.
Specially, for $d=3,4$, since $2\vee\frac{2d(d-1)}{d^2-2d+4}=2$, we not only recover the prior result in \cite{Fan-Killip-Visan-Zhao} for $d=3$, but also derive new dispersive decay estimates for $d=4$. Indeed, fix $d=4$ and $2<r<4$. Given any $u_0\in L^2\cap L^{r'}(\mathbb{R}^4)$ satisfying the hypotheses of Theorem \ref{Dodson}, let $u$ denote the unique global solution to \eqref{NLS} with initial data $u_0$. Then
\begin{equation}\label{d=4}
\sup\limits_{t\neq 0}|t|^{4(\frac{1}{2}-\frac{1}{r})}\|u(t)\|_{L_x^r}\leq C(\|u_0\|_{L_x^2})\|u_0\|_{L_x^{r'}}.
\end{equation}
\end{itemize}
\end{remark}

\section{Linear estimates and useful lemmas}\label{section 2}
This section serves to present some results in the theory of Strichartz estimates and useful lemmas to prove our main theorem.

\begin{definition}
Fix $d\geq 1$. We say that a pair $2\leq p,q\leq \infty$ is Schr\"odinger-admissible if $(d, p, q)\neq (2, 2,\infty)$ and
\begin{equation*}
\frac{2}{p}+\frac{d}{q}=\frac{d}{2}.
\end{equation*}
\end{definition}
In addition to the pointwise-in-time dispersive decay estimate \eqref{classical decay} for solutions to the linear Schr\"odinger equation, a broader framework of space-time estimates—known as \textit{Strichartz estimates}—has been developed. In the pioneering work \cite{R.S.Strichartz}, Strichartz initially established Strichartz estimates for the symmetric pair ($q=p$) based on a Fourier restriction
theorem and then the other cases were subsequently deduced from a combination of an abstract functional analysis argument known as the $TT^*$ duality argument and an $L^1\rightarrow L^\infty$ dispersive estimate, see Keel-Tao \cite{KT}.
\begin{theorem}[\cite{R.S.Strichartz, KT}] \label{Strichartz} Fix $d\geq1$. Let $(p,q)$ and $(\tilde{p},\tilde{q})$ are Schr\"odinger-admissible pairs. Then we have
\begin{align*}
\| e^{it\Delta} u_0 \|_{L_t^pL_x^q(\mathbb{R}\times\mathbb{R}^d)} &\lesssim \| u_0 \|_{L^2(\mathbb{R}^d)}, \\
\left\| \int_{\mathbb{R}} e^{is\Delta} G(s, \cdot) \, ds \right\|_{L^2(\mathbb{R}^d)} &\lesssim \| G \|_{L_t^{p'}L_x^{q'}(\mathbb{R}\times \mathbb{R}^d)}, \\
\left\| \int_0^t e^{i(t-s)\Delta} G(s) \, ds \right\|_{L_t^pL_x^q(\mathbb{R}\times\mathbb{R}^d)}  &\lesssim \|G \|_{L_t^{\tilde{p}'}L_x^{\tilde{q}'}(\mathbb{R}\times\mathbb{R}^d)}.
\end{align*}
\end{theorem}
From the global spacetime bounds in Theorem \eqref{Dodson} and Strichartz estimates \eqref{Strichartz}, we derive the following lemma  which records general Schr\"odinger-admissible spacetime bounds for solutions of \eqref{NLS}.
\begin{lemma} Fix $d\geq1$. Let $(p,q)$ be a Schr\"odinger-admissible pair, that is, $2\leq p,q\leq \infty$, $\frac{2}{p}+\frac{d}{q}=\frac{d}{2}$ and $(d, p, q)\neq (2, 2,\infty)$. Given any $u_0\in L^2(\mathbb{R}^d)$ satisfying the hypotheses of Theorem \ref{Dodson}, let $u$ denote the unique global solution to \eqref{NLS} with initial data $u_0$. Then we have the global spacetime bound
\begin{equation}\label{p,q bound}
\|u\|_{L_t^pL_x^q(\mathbb{R}\times\mathbb{R}^d)}\leq C(\|u_0\|_{L_x^2}).
\end{equation}
\end{lemma}
\begin{proof}
By \eqref{mass-preserved}, it suffices to prove $2\leq p<\infty$. Using the Duhamel formula,
\begin{equation*}
u(t)=e^{it\Delta} u_0-i\int_0^t e^{i(t-s)\Delta} F(u(s))ds,
\end{equation*}
by the Strichartz estimates for $e^{it\Delta}$ in Theorem \ref{Strichartz} and the H\"older inequality, we have
\begin{equation}\label{p-q type}
\begin{aligned}
\|u\|_{L_t^pL_x^q(\mathbb{R}\times\mathbb{R}^d)}&\lesssim \|u_0\|_{L_x^2}+\||u|^\frac{4}{d}u\|_{L_t^{\tilde{p}'}L_x^{\tilde{q}'}(\mathbb{R}\times\mathbb{R}^d)}\\
&\lesssim \|u_0\|_{L_x^2}+\|u\|^\frac{4}{d}_{L_{t,x}^\frac{2(d+2)}{d}(\mathbb{R}\times\mathbb{R}^d)} \|u\|_{L_t^pL_x^q(\mathbb{R}\times\mathbb{R}^d)},
\end{aligned}
\end{equation}
where $(p,q)$ and $(\tilde{p},\tilde{q})$ are Schr\"odinger-admissible pairs satisfying
\begin{equation*}
\frac{1}{\tilde{p}'}=\frac{2}{d+2}+\frac{1}{p} \text{ and } \frac{1}{\tilde{q}'}=\frac{2}{d+2}+\frac{1}{q}.
\end{equation*}
By Theorem \ref{Dodson} and subdividing $\mathbb{R}$ into finitely many subintervals $I_j$ on which
\begin{equation*}
\|u\|_{L_t^pL_x^q(I_j\times\mathbb{R}^d)}<\eta,
\end{equation*}
for $\eta>0$ sufficiently small (depending only on the implicit constant in \eqref{p-q type}), a standard bootstrap argument implies
\begin{equation*}
\|u\|_{L_t^pL_x^q(\mathbb{R}\times\mathbb{R}^d)}\leq C(\|u_0\|_{L_x^2}).
\end{equation*}
\end{proof}
In order to establish the pointwise-in-time dispersive decay in Theorem \ref{main result}, we will employ the following linear estimate to control the contribution of the low power of the nonlinearity in spatial dimensions $d\geq3$.
\begin{lemma}\label{key lemma}
Fix $d\geq3$ and $r\in [2,\infty]$. Let $(p,q)$ be a Schr\"odinger-admissible pair, that is, $2\leq p,q\leq \infty$  and $\frac{2}{p}+\frac{d}{q}=\frac{d}{2}$. For any $t\in (0,\infty)$ and any interval $I\subseteq[0,t]$, we have
\begin{equation}\label{general control}
\begin{aligned}
\left\|\int_{I}e^{i(t-s)\Delta} G(s)ds\right\|_{L_x^r(\mathbb{R}^d)}\lesssim  \left\||t-s|^{-d(\frac{1}{2}-\frac{1}{r})}G(s)\right\|_{L^\frac{pr}{pr-2}_sL_x^\frac{qr}{qr-2}(I\times\mathbb{R}^d)}.
\end{aligned}
\end{equation}
\end{lemma}
\begin{remark} In Lemma \ref{key lemma}, when the Schr\"odinger-admissible pair $(p,q)=(\infty,2)$, the estimate \eqref{general control} reduces to
\begin{equation*}
\left\|\int_{I}e^{i(t-s)\Delta} G(s)ds\right\|_{L_x^r(\mathbb{R}^d)}\lesssim  \left\||t-s|^{-d(\frac{1}{2}-\frac{1}{r})}G(s)\right\|_{L^1_sL_x^\frac{r}{r-1}(I\times\mathbb{R}^d)}.
\end{equation*}
\end{remark}
\begin{remark} Specially, let the spatial dimension $d=3$ and choose $(p,q)=(\infty,2)$ and $(p,q)=(2,6)$ respectively. Then we have
\begin{equation*}
\left\|\int_{I}e^{i(t-s)\Delta} G(s)ds\right\|_{L_x^r(\mathbb{R}^3)}\lesssim  \left\||t-s|^{-d(\frac{1}{2}-\frac{1}{r})}G(s)\right\|_{L^1_sL_x^\frac{r}{r-1}(I\times\mathbb{R}^3)},
\end{equation*}
and
\begin{equation*}
\left\|\int_{I}e^{i(t-s)\Delta} G(s)ds\right\|_{L_x^r(\mathbb{R}^3)}\lesssim  \left\||t-s|^{-d(\frac{1}{2}-\frac{1}{r})}G(s)\right\|_{L^\frac{r}{r-1}_sL_x^\frac{3r}{3r-1}(I\times\mathbb{R}^3)},
\end{equation*}
which recovers Lemma 2.6 in \cite{Fan-Killip-Visan-Zhao}.
\end{remark}
\begin{proof}[Proof of Lemma \ref{key lemma}] It follows immediately from the complex interpolation between Strichartz estimates in Theorem \ref{Strichartz}
\begin{equation*}
\left\| \int_{I} e^{i(t-s)\Delta} G(s) \, ds \right\|_{L^2(\mathbb{R}^d)} \lesssim \|G\|_{L_s^{p'}L_x^{q'}(I\times\mathbb{R}^d)},
\end{equation*}
and the dispersive estimate \eqref{classical decay}
\begin{equation*}
\begin{aligned}
\left\|\int_{I}e^{i(t-s)\Delta} G(s)ds\right\|_{L_x^\infty(\mathbb{R}^d)}&\leq  \int_{I}\|e^{i(t-s)\Delta} G(s)\|_{L_x^\infty(\mathbb{R}^d)}ds\\
&\leq  \int_{I} |t-s|^{-\frac{d}{2}}\|G(s)\|_{L_x^1(\mathbb{R}^d)}ds\\
&\lesssim \left\||t-s|^{-\frac{d}{2}}G(s)\right\|_{L^1_sL_x^1(I\times\mathbb{R}^d)}
\end{aligned}
\end{equation*}
that
\begin{equation*}
\left\|\int_{I}e^{i(t-s)\Delta} G(s)ds\right\|_{L_x^r(\mathbb{R}^d)}\lesssim  \left\||t-s|^{-d(\frac{1}{2}-\frac{1}{r})}G(s)\right\|_{L^{\left(\frac{pr}{2}\right)'}_sL_x^{\left(\frac{qr}{2}\right)'}(I\times\mathbb{R}^d)},
\end{equation*}
which completes the proof of \eqref{general control}.
\end{proof}

\section{Proof of Theorem \ref{main result}}
By time-reversal symmetry, it suffices to prove the pointwise decay estimate for $t\in (0,\infty)$. By the density of Schwartz functions in $L^2\cap L^{r'}(\mathbb{R}^d)$, it suffices to show \eqref{main} for Schwartz solutions to \eqref{NLS}. For $T\in (0,\infty]$, the norm is defined by
\begin{equation}\label{X}
\|u\|_{X(T)}:=\sup\limits_{t\in(0,T)} t^{d(\frac{1}{2}-\frac{1}{r})}\|u(t)\|_{L_x^r(\mathbb{R}^d)}.
\end{equation}
A quick calculation shows that $(2,\frac{2d}{d-2})$ and $(\frac{4r(d+4)}{d[2(d-2)-(d-4)r]},\frac{(d+4)r}{dr-(d-2)})$ are Schr\"odinger-admissible pairs for $2<r<\frac{2d}{d-2}$, by the global spacetime bound \eqref{p,q bound}, we have
\begin{equation}\label{Global}
\|u\|_{L_t^2L_x^\frac{2d}{d-2}(\mathbb{R}\times\mathbb{R}^d)}+\|u\|_{L_t^\frac{4r(d+4)}{d[2(d-2)-(d-4)r]}L_x^\frac{(d+4)r}{dr-(d-2)}(\mathbb{R}\times\mathbb{R}^d)}\leq C(\|u_0\|_{L_x^2}).
\end{equation}
Let $\eta>0$ be a small parameter to be chosen later, depending only on universal constants. By \eqref{Global}, it implies that we may decompose $[0,\infty)$ into $J=J(\|u_0\|_{L_x^2},\eta)$ many intervals $I_j=[T_{j-1},T_j)$, $1\leq j\leq J$ on which
\begin{equation}\label{small eta}
\|u\|_{L_t^2L_x^\frac{2d}{d-2}(I_j\times\mathbb{R}^d)}<\eta \text{ and } \|u\|_{L_t^\frac{4r(d+4)}{d[2(d-2)-(d-4)r]}L_x^\frac{(d+4)r}{dr-(d-2)}(I_j\times\mathbb{R}^d)}<\eta.
\end{equation}
We aim to prove that for each $1\leq j\leq J$, it holds
\begin{equation}\label{bootstrap}
\|u\|_{X(T_j)}\lesssim \|u_0\|_{L_x^{r'}}+C(\|u_0\|_{L_x^2})[\|u\|_{X(T_{j-1})}+\eta^{c(r)}\|u\|_{X(T_j)}],
\end{equation}
where $c(r)>0$ is a constant depending only on $r$. Choosing $\eta>0$ sufficiently small to defeat the absolute implicit constant in \eqref{bootstrap} and $C(\|u_0\|_{L_x^2})$, we readily get
\begin{equation*}
\|u\|_{X(\infty)}\lesssim C(\|u_0\|_{L_x^2}) \|u_0\|_{L_x^{r'}},
\end{equation*}
which yields \eqref{main}.

Now we focus on \eqref{bootstrap}. Fix $1\leq j\leq J$. For any $t\in (0,T_j)$, we decompose
\begin{equation}\label{split}
\begin{aligned}
u(t)=e^{it\Delta}u_0&-i\int_0^{\frac{t}{2}\wedge A}e^{i(t-s)\Delta} F(u(s))ds-i\int_{\frac{t}{2}\wedge A}^\frac{t}{2}e^{i(t-s)\Delta} F(u(s))ds\\
&-i\int_\frac{t}{2}^{\frac{t}{2}\vee (t-B)}e^{i(t-s)\Delta} F(u(s))ds-i\int_{\frac{t}{2}\vee (t-B)}^te^{i(t-s)\Delta} F(u(s))ds,
\end{aligned}
\end{equation}
where $A,B>0$ are constants to be chosen later. In fact, each of these four integrals will be further partitioned into $s\in (0,T_{j-1})$ or $s\in [T_{j-1}, T_j)$. Although we employ the same technique for the two parts, we must separate them because we will eventually choose different parameters $A$ and $B$ for $s\in (0,T_{j-1})$ and $s\in [T_{j-1}, T_j)$, which is essential for obtaining $\eta^{c(r)}$ in \eqref{bootstrap}.
\vskip 0.5cm
By the dispersive estimate \eqref{classical decay}, the contribution of the first term (i.e., the linear term) on the right hand side of \eqref{split} is immediately seen to be acceptable
\begin{equation}
\|e^{it\Delta}u_0\|_{L_x^r}\lesssim |t|^{-d(\frac{1}{2}-\frac{1}{r})}\|u_0\|_{L_x^{r'}}.
\end{equation}
Therefore, we employ Lemma \ref{key lemma} to estimate the contribution of nonlinearity
on the right side hand of \eqref{split}.  Note that $|t-s|\sim |t|$ whenever $s\in [0,\frac{t}{2}]$ and $|t|\sim |s|$ whenever $s\in [\frac{t}{2},t]$.

For the second term, we obtain
\begin{equation}\label{Second term}
\begin{aligned}
\left\|\int_0^{\frac{t}{2}\wedge A}e^{i(t-s)\Delta} F(u(s))ds\right\|_{L_x^r}&\lesssim  \left\||t-s|^{-d(\frac{1}{2}-\frac{1}{r})}|u(s)|^\frac{d+4}{d}\right\|_{L^\frac{pr}{pr-2}_sL_x^\frac{qr}{qr-2}([0,\frac{t}{2}\wedge A]\times\mathbb{R}^d)}\\
&\lesssim |t|^{-d(\frac{1}{2}-\frac{1}{r})}\|u\|^\frac{d+4}{d}_{L^{\frac{pr}{pr-2}\cdot \frac{d+4}{d}}_sL_x^{\frac{qr}{qr-2}\cdot\frac{d+4}{d}}([0,\frac{t}{2}\wedge A]\times\mathbb{R}^d)},
\end{aligned}
\end{equation}
where $(p,q)$ is a Schr\"odinger-admissible pair. Here we choose $(p,q)=(2,\frac{2d}{d-2})$ so that $2\leq \frac{qr}{qr-2}\cdot\frac{d+4}{d}=\frac{(d+4)r}{dr-(d-2)}\leq \frac{2d}{d-2}$ and then $\frac{pr}{pr-2}\cdot \frac{d+4}{d}=\frac{r(d+4)}{d(r-1)}$. By
\begin{equation*}
\frac{d(r-1)}{r(d+4)}=\frac{d[2(d-2)-(d-4)r]}{4r(d+4)}+\frac{d^2(r-2)}{4r(d+4)},
\end{equation*}
along with H\"older's inequality in \eqref{Second term}, we obtain
\begin{equation}\label{second term}
\begin{aligned}
\left\|\int_0^{\frac{t}{2}\wedge A}e^{i(t-s)\Delta} F(u(s))ds\right\|_{L_x^r}
&\lesssim |t|^{-d(\frac{1}{2}-\frac{1}{r})}\|u\|^\frac{d+4}{d}_{L^\frac{r(d+4)}{d(r-1)}_sL_x^\frac{(d+4)r}{dr-(d-2)}([0,\frac{t}{2}\wedge A]\times\mathbb{R}^d)}\\
&\lesssim |t|^{-d(\frac{1}{2}-\frac{1}{r})}A^\frac{d(r-2)}{4r}\|u\|^\frac{d+4}{d}_{L^\frac{4r(d+4)}{d[2(d-2)-(d-4)r]}_sL_x^\frac{(d+4)r}{dr-(d-2)}([0,\frac{t}{2}\wedge A]\times\mathbb{R}^d)}.
\end{aligned}
\end{equation}
From \eqref{Global} and \eqref{second term}, we have
\begin{equation}\label{second P1}
\left\|\int_{[0,\frac{t}{2}\wedge A]\cap [0,T_{j-1})}e^{i(t-s)\Delta} F(u(s))ds\right\|_{L_x^r}\lesssim |t|^{-d(\frac{1}{2}-\frac{1}{r})}A^\frac{d(r-2)}{4r}C(\|u_0\|_{L_x^2}).
\end{equation}
From \eqref{small eta} and \eqref{second term}, we get
\begin{equation}\label{second P2}
\left\|\int_{[0,\frac{t}{2}\wedge A]\cap [T_{j-1},T_j)}e^{i(t-s)\Delta} F(u(s))ds\right\|_{L_x^r}\lesssim |t|^{-d(\frac{1}{2}-\frac{1}{r})}A^\frac{d(r-2)}{4r}\eta^\frac{d+4}{d}.
\end{equation}
For the fourth term,  by parallel methods, it follows
\begin{equation}\label{fourth term}
\begin{aligned}
&\quad\left\|\int_\frac{t}{2}^{\frac{t}{2}\vee (t-B)}e^{i(t-s)\Delta} F(u(s))ds\right\|_{L_x^r}\\
&\lesssim  \left\||t-s|^{-d(\frac{1}{2}-\frac{1}{r})}|u(s)|^\frac{d+4}{d}\right\|_{L^\frac{r}{r-1}_sL_x^\frac{dr}{dr-(d-2)}([\frac{t}{2},\frac{t}{2}\vee (t-B)]\times\mathbb{R}^d)}\\
&=  \left\||t-s|^{-\frac{d^2}{d+4}(\frac{1}{2}-\frac{1}{r})}|u(s)|\right\|^\frac{d+4}{d}_{L^\frac{r(d+4)}{d(r-1)}_sL_x^\frac{(d+4)r}{dr-(d-2)}([\frac{t}{2},\frac{t}{2}\vee (t-B)]\times\mathbb{R}^d)}\\
&\lesssim \left\||t-s|^{-\frac{d^2}{d+4}(\frac{1}{2}-\frac{1}{r})}\right\|^\frac{d+4}{d}_{L_s^\frac{4r(d+4)}{d^2(r-2)}[\frac{t}{2},\frac{t}{2}\vee (t-B)]}\|u\|^\frac{d+4}{d}_{L^\frac{4r(d+4)}{d[2(d-2)-(d-4)r]}_sL_x^\frac{(d+4)r}{dr-(d-2)}([\frac{t}{2},\frac{t}{2}\vee (t-B)]\times\mathbb{R}^d)}\\
&=\left(\int_\frac{t}{2}^{\frac{t}{2}\vee (t-B)}|t-s|^{-2}ds\right)^\frac{d(r-2)}{4r}\|u\|^\frac{d+4}{d}_{L^\frac{4r(d+4)}{d[2(d-2)-(d-4)r]}_sL_x^\frac{(d+4)r}{dr-(d-2)}([\frac{t}{2},\frac{t}{2}\vee (t-B)]\times\mathbb{R}^d)}\\
&\lesssim B^{-\frac{d(r-2)}{4r}}\|u\|^\frac{d+4}{d}_{L^\frac{4r(d+4)}{d[2(d-2)-(d-4)r]}_sL_x^\frac{(d+4)r}{dr-(d-2)}([\frac{t}{2},\frac{t}{2}\vee (t-B)]\times\mathbb{R}^d)}.
\end{aligned}
\end{equation}
From \eqref{Global} and \eqref{fourth term}, we have
\begin{equation}\label{fourth P1}
\left\|\int_{[\frac{t}{2},\frac{t}{2}\vee (t-B)]\cap [0,T_{j-1})}e^{i(t-s)\Delta} F(u(s))ds\right\|_{L_x^r}\lesssim B^{-\frac{d(r-2)}{4r}}C(\|u_0\|_{L_x^2}).
\end{equation}
From \eqref{small eta} and \eqref{fourth term}, we get
\begin{equation}\label{fourth P2}
\left\|\int_{[\frac{t}{2},\frac{t}{2}\vee (t-B)]\cap [T_{j-1},T_j)}e^{i(t-s)\Delta} F(u(s))ds\right\|_{L_x^r}\lesssim B^{-\frac{d(r-2)}{4r}}\eta^\frac{d+4}{d}.
\end{equation}
Now we turn to the third term and the last term. As \eqref{Second term}, we obtain
\begin{equation}\label{Third term}
\begin{aligned}
\left\|\int_{\frac{t}{2}\wedge A}^\frac{t}{2} e^{i(t-s)\Delta} F(u(s))ds\right\|_{L_x^r}&\lesssim  \left\||t-s|^{-d(\frac{1}{2}-\frac{1}{r})}|u(s)|^\frac{d+4}{d}\right\|_{L^\frac{pr}{pr-2}_sL_x^\frac{qr}{qr-2}([\frac{t}{2}\wedge A,\frac{t}{2}]\times\mathbb{R}^d)}\\
&\lesssim |t|^{-d(\frac{1}{2}-\frac{1}{r})}\left\|\|u(s)\|_{L_x^{\frac{qr}{qr-2}\cdot\frac{d+4}{d}}}\right\|^\frac{d+4}{d}_{L^{\frac{pr}{pr-2}\cdot \frac{d+4}{d}}_s[\frac{t}{2}\wedge A,\frac{t}{2}]},
\end{aligned}
\end{equation}
and
\begin{equation}\label{Last term}
\begin{aligned}
\left\|\int_{\frac{t}{2}\vee (t-B)}^te^{i(t-s)\Delta} F(u(s))ds\right\|_{L_x^r}&\lesssim  \left\||t-s|^{-d(\frac{1}{2}-\frac{1}{r})}|u(s)|^\frac{d+4}{d}\right\|_{L^\frac{pr}{pr-2}_sL_x^\frac{qr}{qr-2}([\frac{t}{2}\vee (t-B),t]\times\mathbb{R}^d)}\\
&\lesssim  \left\||t-s|^{-d(\frac{1}{2}-\frac{1}{r})}\|u(s)\|^\frac{d+4}{d}_{L_x^{\frac{qr}{qr-2}\cdot \frac{d+4}{d}}}\right\|_{L^{\frac{pr}{pr-2}}_s[\frac{t}{2}\vee (t-B),t]},
\end{aligned}
\end{equation}
where $(p,q)$ is a Schr\"odinger-admissible pair. We divide $r$ into three subcases: $\frac{2d(d-1)}{d^2-2d+4}\vee 2<r<\frac{2(d+1)}{d}$, $\frac{2(d+1)}{d}\leq r\leq \frac{2(d+2)}{d}$ and $\frac{2(d+2)}{d}<r<\frac{2d}{d-2}$. Note that when $\frac{2d(d-1)}{d^2-2d+4}\vee 2<r<\frac{2(d+1)}{d}$, $\frac{qr}{qr-2}\cdot\frac{d+4}{d}>r$ holds for any Schr\"odinger-admissible pair $(p,q)$; when $\frac{2(d+2)}{d}<r<\frac{2d}{d-2}$, $\frac{qr}{qr-2}\cdot\frac{d+4}{d}<r$ holds for any Schr\"odinger-admissible pair $(p,q)$; when $\frac{2(d+1)}{d}\leq r\leq \frac{2(d+2)}{d}$, there exists a unique Schr\"odinger-admissible pair $(p,q)$ such that $\frac{qr}{qr-2}\cdot\frac{d+4}{d}=r$.\\

\textbf{\underline{Case 1.}} $\frac{2d(d-1)}{d^2-2d+4}\vee 2<r<\frac{2(d+1)}{d}$. In this case, we take $(p,q)=(2,\frac{2d}{d-2})$ and we can see $r<\frac{qr}{qr-2}\cdot\frac{d+4}{d}=\frac{(d+4)r}{dr-(d-2)}<\frac{2d}{d-2}$ and $\frac{pr}{pr-2}\cdot \frac{d+4}{d}=\frac{r(d+4)}{d(r-1)}$. Then we have
\begin{equation}\label{Holder 1}
\|u(s)\|_{L_x^\frac{(d+4)r}{dr-(d-2)}}\leq \|u(s)\|^{1-\theta}_{L_x^\frac{2d}{d-2}}\|u(s)\|^\theta_{L_x^r},
\end{equation}
where
\begin{equation*}
\theta=\frac{(d^2-2d+8)r-2d(d-2)}{(d+4)[2d-(d-2)r]} \text{ and } 1-\theta=\frac{2d[2(d+1)-dr]}{(d+4)[2d-(d-2)r]}.
\end{equation*}
For the third term, it follows from \eqref{X}, \eqref{Third term} and \eqref{Holder 1} that
\begin{equation}\label{third term}
\begin{aligned}
&\quad\left\|\int_{[\frac{t}{2}\wedge A, \frac{t}{2}]\cap [0,T_{j-1})} e^{i(t-s)\Delta} F(u(s))ds\right\|_{L_x^r}\\
&\lesssim |t|^{-d(\frac{1}{2}-\frac{1}{r})}\left\|\|u(s)\|_{L_x^{\frac{(d+4)r}{dr-(d-2)}}}\right\|^\frac{d+4}{d}_{L^{\frac{(d+4)r}{d(r-1)}}_s\left([\frac{t}{2}\wedge A, \frac{t}{2}]\cap [0,T_{j-1})\right)}\\
&\leq |t|^{-d(\frac{1}{2}-\frac{1}{r})}\left\|\|u(s)\|^{1-\theta}_{L_x^\frac{2d}{d-2}}\|u(s)\|^\theta_{L_x^r}\right\|^\frac{d+4}{d}_{L^{\frac{(d+4)r}{d(r-1)}}_s\left([\frac{t}{2}\wedge A, \frac{t}{2}]\cap [0,T_{j-1})\right)}\\
&\lesssim |t|^{-d(\frac{1}{2}-\frac{1}{r})}\|u\|^{\frac{d+4}{d}\theta}_{X(T_{j-1})}\left\|\|u(s)\|^{1-\theta}_{L_x^\frac{2d}{d-2}}|s|^{-d(\frac{1}{2}-\frac{1}{r})\theta}\right\|^\frac{d+4}{d}_{L^{\frac{(d+4)r}{d(r-1)}}_s\left([\frac{t}{2}\wedge A, \frac{t}{2}]\cap [0,T_{j-1})\right)}\\
&=|t|^{-d(\frac{1}{2}-\frac{1}{r})}\|u\|^{\frac{d+4}{d}\theta}_{X(T_{j-1})}\left\|\|u(s)\|_{L_x^\frac{2d}{d-2}}|s|^{-d(\frac{1}{2}-\frac{1}{r})\frac{\theta}{{1-\theta}}}\right\|^{\frac{d+4}{d}(1-\theta)}_{L^{\frac{(d+4)r}{d(r-1)}(1-\theta)}_s\left([\frac{t}{2}\wedge A, \frac{t}{2}]\cap [0,T_{j-1})\right)}.
\end{aligned}
\end{equation}
Note that
\begin{equation*}
2-\frac{(d+4)r}{d(r-1)}(1-\theta)=2-\frac{2r[2(d+1)-dr]}{(r-1)[2d-(d-2)r]}=\frac{2(r-2)(2r+d)}{(r-1)[2d-(d-2)r]}>0,
\end{equation*}
which means that $0<\frac{(d+4)r}{d(r-1)}(1-\theta)<2$. Let $\alpha=\frac{2r[2(d+1)-dr]}{(r-2)(2r+d)}>0$ satisfy
\begin{equation*}
\frac{1}{\frac{(d+4)r}{d(r-1)}(1-\theta)}=\frac{1}{2}+\frac{1}{\alpha}
\end{equation*}
and $\alpha d(\frac{1}{2}-\frac{1}{r})\frac{\theta}{{1-\theta}}-1=\frac{(d^2-2d+4)r-2d(d-1)}{2(2r+d)}>0$, i.e., $\alpha d(\frac{1}{2}-\frac{1}{r})\frac{\theta}{{1-\theta}}>1$.
By \eqref{Global}, using H\"older's inequality in \eqref{third term}, it yields
\begin{equation}\label{case-1 p1}
\begin{aligned}
&\quad\left\|\int_{[\frac{t}{2}\wedge A, \frac{t}{2}]\cap [0,T_{j-1})} e^{i(t-s)\Delta} F(u(s))ds\right\|_{L_x^r}\\
&\lesssim |t|^{-d(\frac{1}{2}-\frac{1}{r})}\|u\|^{\frac{d+4}{d}\theta}_{X(T_{j-1})}\|u\|^{\frac{d+4}{d}(1-\theta)}_{L_s^2L_x^\frac{2d}{d-2}\left(\left([\frac{t}{2}\wedge A, \frac{t}{2}]\cap [0,T_{j-1})\right)\times\mathbb{R}^d\right)}\left\||s|^{-d(\frac{1}{2}-\frac{1}{r})\frac{\theta}{{1-\theta}}}\right\|^{\frac{d+4}{d}(1-\theta)}_{L^\alpha_s\left([\frac{t}{2}\wedge A, \frac{t}{2}]\cap [0,T_{j-1})\right)}\\
&\lesssim |t|^{-d(\frac{1}{2}-\frac{1}{r})}\|u\|^{\frac{d+4}{d}\theta}_{X(T_{j-1})}\|u\|^{\frac{d+4}{d}(1-\theta)}_{L_s^2L_x^\frac{2d}{d-2}\left(\left([\frac{t}{2}\wedge A, \frac{t}{2}]\cap [0,T_{j-1})\right)\times\mathbb{R}^d\right)}\left\||s|^{-d(\frac{1}{2}-\frac{1}{r})\frac{\theta}{{1-\theta}}}\right\|^{\frac{d+4}{d}(1-\theta)}_{L^\alpha_s\left([\frac{t}{2}\wedge A, \frac{t}{2}]\cap [0,T_{j-1})\right)}\\
&\lesssim |t|^{-d(\frac{1}{2}-\frac{1}{r})}\|u\|^{\frac{d+4}{d}\theta}_{X(T_{j-1})}\|u\|^{\frac{d+4}{d}(1-\theta)}_{L_s^2L_x^\frac{2d}{d-2}\left(\left([\frac{t}{2}\wedge A, \frac{t}{2}]\cap [0,T_{j-1})\right)\times\mathbb{R}^d\right)}\left(\int_{\frac{t}{2}\wedge A}^\frac{t}{2}|s|^{-\alpha d(\frac{1}{2}-\frac{1}{r})\frac{\theta}{{1-\theta}}}ds\right)^{\frac{d+4}{d\alpha}(1-\theta)}\\
&\lesssim |t|^{-d(\frac{1}{2}-\frac{1}{r})}\|u\|^{\frac{d+4}{d}\theta}_{X(T_{j-1})}\|u\|^{\frac{d+4}{d}(1-\theta)}_{L_s^2L_x^\frac{2d}{d-2}\left(\left([\frac{t}{2}\wedge A, \frac{t}{2}]\cap [0,T_{j-1})\right)\times\mathbb{R}^d\right)}A^{-\frac{(d^2-2d+4)r-2d(d-1)}{2(2r+d)}\cdot \frac{d+4}{d\alpha}(1-\theta)}\\
&\lesssim |t|^{-d(\frac{1}{2}-\frac{1}{r})}\|u\|^{\frac{(d^2-2d+8)r-2d(d-2)}{d[2d-(d-2)r]} }_{X(T_{j-1})}\|u\|^{\frac{2[2(d+1)-dr]}{2d-(d-2)r}}_{L_s^2L_x^\frac{2d}{d-2}\left([0,T_{j-1})\times\mathbb{R}^d\right)}A^{-\frac{(r-2)[(d^2-2d+4)r-2d(d-1)]}{2r[2d-(d-2)r]}}\\
&\lesssim |t|^{-d(\frac{1}{2}-\frac{1}{r})}\|u\|^{\frac{(d^2-2d+8)r-2d(d-2)}{d[2d-(d-2)r]} }_{X(T_{j-1})}A^{-\frac{(r-2)[(d^2-2d+4)r-2d(d-1)]}{2r[2d-(d-2)r]}} C(\|u_0\|_{L^2_x}).
\end{aligned}
\end{equation}
Argued analogously as \eqref{case-1 p1}, by \eqref{small eta}, we also have
\begin{equation}\label{case-1 p2}
\begin{aligned}
&\quad\left\|\int_{[\frac{t}{2}\wedge A, \frac{t}{2}]\cap [T_{j-1},T_j)} e^{i(t-s)\Delta} F(u(s))ds\right\|_{L_x^r}\\
&\lesssim |t|^{-d(\frac{1}{2}-\frac{1}{r})}\|u\|^{\frac{(d^2-2d+8)r-2d(d-2)}{d[2d-(d-2)r]} }_{X(T_j)}\|u\|^{\frac{2[2(d+1)-dr]}{2d-(d-2)r}}_{L_s^2L_x^\frac{2d}{d-2}\left([T_{j-1},T_j)\times\mathbb{R}^d\right)}A^{-\frac{(r-2)[(d^2-2d+4)r-2d(d-1)]}{2r[2d-(d-2)r]}}\\
&\lesssim |t|^{-d(\frac{1}{2}-\frac{1}{r})}\|u\|^{\frac{(d^2-2d+8)r-2d(d-2)}{d[2d-(d-2)r]} }_{X(T_j)}A^{-\frac{(r-2)[(d^2-2d+4)r-2d(d-1)]}{2r[2d-(d-2)r]}} \eta^{\frac{2[2(d+1)-dr]}{2d-(d-2)r}}.
\end{aligned}
\end{equation}

Now we optimize in $A$ the bounds which are obtained for the second and third terms on the right side hand \eqref{split}. Choose $A=\|u\|_{X(T_{j-1})}^\frac{4r}{d(r-2)}$ to minimize the sum of the right hand of \eqref{second P1} and \eqref{case-1 p1}, which yields
\begin{equation}\label{case 1-1}
\left\|\int_{[0,\frac{t}{2}]\cap [0,T_{j-1})}e^{i(t-s)\Delta}F(u(s))\right\|_{L_x^r}\lesssim |t|^{-d(\frac{1}{2}-\frac{1}{r})} C(\|u_0\|_{L_x^2})\|u\|_{X(T_{j-1})}.
\end{equation}
Similarly, choose $A=\left(\|u\|_{X(T_j)}\eta^{-1}\right)^\frac{4r}{d(r-2)}$ to minimize the sum of the right hand side of \eqref{second P2} and  \eqref{case-1 p2}, which yields
\begin{equation}\label{case 1-2}
\left\|\int_{[0,\frac{t}{2}]\cap [T_{j-1},T_j)}e^{i(t-s)\Delta}F(u(s))\right\|_{L_x^r} \lesssim |t|^{-d(\frac{1}{2}-\frac{1}{r})} \eta^\frac{4}{d}\|u\|_{X(T_j)}.
\end{equation}

It follows from \eqref{case 1-1} and \eqref{case 1-2} that
\begin{equation}\label{case 1}
\left\|\int_0^\frac{t}{2}e^{i(t-s)\Delta}F(u(s))\right\|_{L_x^r}\lesssim |t|^{-d(\frac{1}{2}-\frac{1}{r})} C(\|u_0\|_{L_x^2})\left(\|u\|_{X(T_{j-1})}+\eta^{c(r)}\|u\|_{X(T_j)}\right),
\end{equation}
with $c(r)=\frac{4}{d}$.

For the last term, by \eqref{X}, \eqref{Global}, \eqref{Last term} and \eqref{Holder 1} , we have
\begin{equation}\label{Case-1 P1}
\begin{aligned}
&\quad\left\|\int_{[\frac{t}{2}\vee (t-B),t]\cap [0,T_{j-1})} e^{i(t-s)\Delta} F(u(s))ds\right\|_{L_x^r}\\
&\lesssim \left\||t-s|^{-d(\frac{1}{2}-\frac{1}{r})}\|u(s)\|^\frac{d+4}{d}_{L_x^{\frac{(d+4)r}{dr-(d-2)}}}\right\|_{L^{\frac{r}{r-1}}_s\left([\frac{t}{2}\vee (t-B),t]\cap [0,T_{j-1})\right)}\\
&\leq \left\||t-s|^{-d(\frac{1}{2}-\frac{1}{r})}\|u(s)\|^{\frac{d+4}{d}(1-\theta)}_{L_x^\frac{2d}{d-2}}\|u(s)\|^{\frac{d+4}{d}\theta}_{L_x^r}\right\|_{L^{\frac{r}{r-1}}_s\left([\frac{t}{2}\vee (t-B),t]\cap [0,T_{j-1})\right)}\\
&\leq \left\||t-s|^{-d(\frac{1}{2}-\frac{1}{r})}\|u(s)\|^{\frac{d+4}{d}(1-\theta)}_{L_x^\frac{2d}{d-2}}\|u\|^{\frac{d+4}{d}\theta}_{X(T_{j-1})}|s|^{-(d+4)(\frac{1}{2}-\frac{1}{r})\theta}\right\|_{L^{\frac{r}{r-1}}_s\left([\frac{t}{2}\vee (t-B),t]\cap [0,T_{j-1})\right)}\\
&\lesssim|t|^{-(d+4)(\frac{1}{2}-\frac{1}{r})\theta}\|u\|^{\frac{d+4}{d}\theta}_{X(T_{j-1})}\left\|\|u(s)\|_{L_x^\frac{2d}{d-2}}|t-s|^{-\frac{d^2}{(d+4)(1-\theta)}(\frac{1}{2}-\frac{1}{r})}\right\|^{\frac{d+4}{d}(1-\theta)}_{L^{\frac{(d+4)r}{d(r-1)}(1-\theta)}_s\left([\frac{t}{2}\vee (t-B),t]\cap [0,T_{j-1})\right)}\\
&\leq |t|^{-(d+4)(\frac{1}{2}-\frac{1}{r})\theta}\|u\|^{\frac{d+4}{d}\theta}_{X(T_{j-1})}\|u\|^{\frac{d+4}{d}(1-\theta)}_{L_s^2L_x^\frac{2d}{d-2}\left([0,T_{j-1})\times\mathbb{R}^d\right)}\left\||t-s|^{-\frac{d^2}{(d+4)(1-\theta)}(\frac{1}{2}-\frac{1}{r})}\right\|^{\frac{d+4}{d}(1-\theta)}_{L^\alpha_s[\frac{t}{2}\vee (t-B),t]}\\
&= |t|^{-(d+4)(\frac{1}{2}-\frac{1}{r})\theta}\|u\|^{\frac{d+4}{d}\theta}_{X(T_{j-1})}\|u\|^{\frac{d+4}{d}(1-\theta)}_{L_s^2L_x^\frac{2d}{d-2}\left([0,T_{j-1})\times\mathbb{R}^d\right)}\left(\int_{\frac{t}{2}\vee (t-B)}^t|t-s|^{^{-\frac{d^2\alpha}{(d+4)(1-\theta)}(\frac{1}{2}-\frac{1}{r})}}ds\right)^{\frac{d+4}{d\alpha}(1-\theta)}\\
&=|t|^{-\frac{(r-2)[(d^2-2d+8)r-2d(d-2)]}{2r[2d-(d-2)r]}}\|u\|^{\frac{(d^2-2d+8)r-2d(d-2)}{d[2d-(d-2)r]} }_{X(T_{j-1})}\|u\|^{\frac{2[2(d+1)-dr]}{2d-(d-2)r}}_{L_s^2L_x^\frac{2d}{d-2}\left([0,T_{j-1})\times\mathbb{R}^d\right)}\\
&\qquad \qquad\qquad \qquad\times\left(\int_{\frac{t}{2}\vee (t-B)}^t|t-s|^{\frac{(d^2-2d+4)r-2d(d-1)}{2(2r+d)}-1}ds\right)^{\frac{(r-2)(2r+d)}{r[2d-(d-2)r]}}\\
&\lesssim|t|^{-\frac{(r-2)[(d^2-2d+8)r-2d(d-2)]}{2r[2d-(d-2)r]}}\|u\|^{\frac{(d^2-2d+8)r-2d(d-2)}{d[2d-(d-2)r]} }_{X(T_{j-1})}\|u\|^{\frac{2[2(d+1)-dr]}{2d-(d-2)r}}_{L_s^2L_x^\frac{2d}{d-2}\left([0,T_{j-1})\times\mathbb{R}^d\right)} B^{\frac{(r-2)[(d^2-2d+4)r-2d(d-1)]}{2r[2d-(d-2)r]}}\\
&\leq |t|^{-\frac{(r-2)[(d^2-2d+8)r-2d(d-2)]}{2r[2d-(d-2)r]}}\|u\|^{\frac{(d^2-2d+8)r-2d(d-2)}{d[2d-(d-2)r]} }_{X(T_{j-1})}B^{\frac{(r-2)[(d^2-2d+4)r-2d(d-1)]}{2r[2d-(d-2)r]}} C(\|u_0\|_{L^2_x}).
\end{aligned}
\end{equation}
Argued analogously as \eqref{Case-1 P1}, by \eqref{small eta}, we also have
\begin{equation}\label{Case-1 P2}
\begin{aligned}
&\quad\left\|\int_{[\frac{t}{2}\vee (t-B),t]\cap [T_{j-1},T_j)} e^{i(t-s)\Delta} F(u(s))ds\right\|_{L_x^r}\\
&\leq |t|^{-\frac{(r-2)[(d^2-2d+8)r-2d(d-2)]}{2r[2d-(d-2)r]}}\|u\|^{\frac{(d^2-2d+8)r-2d(d-2)}{d[2d-(d-2)r]} }_{X(T_{j})}\|u\|^{\frac{2[2(d+1)-dr]}{2d-(d-2)r}}_{L_s^2L_x^\frac{2d}{d-2}\left([T_{j-1},T_j)\times\mathbb{R}^d\right)} B^{\frac{(r-2)[(d^2-2d+4)r-2d(d-1)]}{2r[2d-(d-2)r]}}\\
&\leq |t|^{-\frac{(r-2)[(d^2-2d+8)r-2d(d-2)]}{2r[2d-(d-2)r]}}\|u\|^{\frac{(d^2-2d+8)r-2d(d-2)}{d[2d-(d-2)r]} }_{X(T_{j})}B^{\frac{(r-2)[(d^2-2d+4)r-2d(d-1)]}{2r[2d-(d-2)r]}} \eta^{\frac{2[2(d+1)-dr]}{2d-(d-2)r}}.
\end{aligned}
\end{equation}

Finally, we choose $B$ to minimize the bounds obtained for the fourth and last terms on the right side hand \eqref{split}. Choose $B=|t|^2\|u\|_{X(T_{j-1})}^{-\frac{4r}{d(r-2)}}$ to minimize the sum of the right hand side of \eqref{fourth P1} and \eqref{Case-1 P1}, which yields
\begin{equation}\label{Case-1-1}
\left\|\int_{[\frac{t}{2},t]\cap [0,T_{j-1})} e^{i(t-s)\Delta} F(u(s))ds\right\|_{L_x^r}\lesssim |t|^{-d(\frac{1}{2}-\frac{1}{r})} C(\|u_0\|_{L_x^2})\|u\|_{X(T_{j-1})}.
\end{equation}
Similarly, choose $B=|t|^2\left(\|u\|_{X(T_j)}\eta^{-1}\right)^{-\frac{4r}{d(r-2)}}$ to minimize the sum of the right hand side of \eqref{fourth P2} and \eqref{Case-1 P2}, which yields
\begin{equation}\label{Case-1-2}
\left\|\int_{[\frac{t}{2},t]\cap [T_{j-1},T_j)} e^{i(t-s)\Delta} F(u(s))ds\right\|_{L_x^r}\lesssim |t|^{-d(\frac{1}{2}-\frac{1}{r})} \eta^\frac{4}{d}\|u\|_{X(T_j)}.
\end{equation}
Collecting the bounds \eqref{Case-1-1} and \eqref{Case-1-2}, we have
\begin{equation}\label{Case 1}
\left\|\int_\frac{t}{2}^te^{i(t-s)\Delta}F(u(s))\right\|_{L_x^r}\lesssim |t|^{-d(\frac{1}{2}-\frac{1}{r})} C(\|u_0\|_{L_x^2})\left(\|u\|_{X(T_{j-1})}+\eta^{c(r)}\|u\|_{X(T_j)}\right),
\end{equation}
with $c(r)=\frac{4}{d}$.

\textbf{\underline{Case 2.}} $\frac{2(d+1)}{d}\leq r\leq\frac{2(d+2)}{d}$.
In this case, we take $(p,q)=(\frac{4}{2(d+2)-dr},\frac{2d}{dr-(d+4)})$ and we can see $\frac{qr}{qr-2}\cdot\frac{d+4}{d}=r$ and $\frac{pr}{pr-2}=\frac{2r}{(d+2)(r-2)}$.
For the third term, it follows from \eqref{X} and \eqref{Third term} that
\begin{equation}\label{case-2 p1}
\begin{aligned}
&\quad\left\|\int_{[\frac{t}{2}\wedge A, \frac{t}{2}]\cap [0,T_{j-1})} e^{i(t-s)\Delta} F(u(s))ds\right\|_{L_x^r}\\
&\lesssim |t|^{-d(\frac{1}{2}-\frac{1}{r})}\left\|\|u(s)\|_{L_x^r}\right\|^\frac{d+4}{d}_{L^\frac{2(d+4)r}{d(d+2)(r-2)}_s\left([\frac{t}{2}\wedge A, \frac{t}{2}]\cap [0,T_{j-1})\right)}\\
&\lesssim |t|^{-d(\frac{1}{2}-\frac{1}{r})}\|u\|^{\frac{d+4}{d}}_{X(T_{j-1})}\left\||s|^{-d(\frac{1}{2}-\frac{1}{r})}\right\|^\frac{d+4}{d}_{L^{\frac{2(d+4)r}{d(d+2)(r-2)}}_s\left([\frac{t}{2}\wedge A, \frac{t}{2}]\cap [0,T_{j-1})\right)}\\
&\lesssim |t|^{-d(\frac{1}{2}-\frac{1}{r})}\|u\|^{\frac{d+4}{d}}_{X(T_{j-1})}\left(\int_{\frac{t}{2}\wedge A}^\frac{t}{2}|s|^{-\frac{2}{d+2}-1} ds\right)^\frac{(d+2)(r-2)}{2r}\\
&\lesssim |t|^{-d(\frac{1}{2}-\frac{1}{r})}\|u\|^{\frac{d+4}{d}}_{X(T_{j-1})}A^{-\frac{r-2}{r}}.
\end{aligned}
\end{equation}
Argued analogously as \eqref{case-2 p1}, we also have
\begin{equation}\label{case-2 p2}
\left\|\int_{[\frac{t}{2}\wedge A, \frac{t}{2}]\cap [T_{j-1},T_j)} e^{i(t-s)\Delta} F(u(s))ds\right\|_{L_x^r}\lesssim |t|^{-d(\frac{1}{2}-\frac{1}{r})}\|u\|^{\frac{d+4}{d}}_{X(T_j)}A^{-\frac{r-2}{r}}.
\end{equation}

Choose $A=\|u\|_{X(T_{j-1})}^\frac{4r}{d(r-2)}$ to minimize the sum of the right hand side of \eqref{second P1} and \eqref{case-2 p1}, which yields
\begin{equation}\label{case-2-1}
\left\|\int_{[0,\frac{t}{2}]\cap [0,T_{j-1})}e^{i(t-s)\Delta}F(u(s))\right\|_{L_x^r}\lesssim |t|^{-d(\frac{1}{2}-\frac{1}{r})} C(\|u_0\|_{L_x^2})\|u\|_{X(T_{j-1})}.
\end{equation}
Similarly, choose $A=\left(\|u\|_{X(T_j)}\eta^{-1}\right)^\frac{4r}{d(r-2)}$ to minimize the sum of the right hand side of \eqref{second P2} and  \eqref{case-2 p2}, which yields
\begin{equation}\label{case-2-2}
\left\|\int_{[0,\frac{t}{2}]\cap [T_{j-1},T_j)}e^{i(t-s)\Delta}F(u(s))\right\|_{L_x^r}\lesssim |t|^{-d(\frac{1}{2}-\frac{1}{r})} \eta^\frac{4}{d}\|u\|_{X(T_j)}.
\end{equation}

It follows from \eqref{case-2-1} and \eqref{case-2-2} that
\begin{equation}\label{case 2}
\left\|\int_0^\frac{t}{2}e^{i(t-s)\Delta}F(u(s))\right\|_{L_x^r}\lesssim |t|^{-d(\frac{1}{2}-\frac{1}{r})} C(\|u_0\|_{L_x^2})\left(\|u\|_{X(T_{j-1})}+\eta^{c(r)}\|u\|_{X(T_j)}\right),
\end{equation}
with $c(r)=\frac{4}{d}$.

For the last term, it follows from \eqref{X} and \eqref{Last term} that
\begin{equation}\label{Case-2 P1}
\begin{aligned}
&\quad\left\|\int_{[\frac{t}{2}\vee (t-B),t]\cap [0,T_{j-1})} e^{i(t-s)\Delta} F(u(s))ds\right\|_{L_x^r}\\
&\lesssim \left\||t-s|^{-d(\frac{1}{2}-\frac{1}{r})}\|u(s)\|^\frac{d+4}{d}_{L_x^r}\right\|_{L^\frac{2r}{(d+2)(r-2)}_s\left([\frac{t}{2}\vee (t-B),t]\cap [0,T_{j-1})\right)}\\
&\lesssim \left\||t-s|^{-d(\frac{1}{2}-\frac{1}{r})}\|u\|^{\frac{d+4}{d}}_{X(T_{j-1})}|s|^{-(d+4)(\frac{1}{2}-\frac{1}{r})}\right\|_{L^\frac{2r}{(d+2)(r-2)}_s\left([\frac{t}{2}\vee (t-B),t]\cap [0,T_{j-1})\right)}\\
&\lesssim |t|^{-(d+4)(\frac{1}{2}-\frac{1}{r})}\|u\|^{\frac{d+4}{d}}_{X(T_{j-1})}\left\||t-s|^{-d(\frac{1}{2}-\frac{1}{r})}\right\|_{L^\frac{2r}{(d+2)(r-2)}_s\left([\frac{t}{2}\vee (t-B),t]\cap [0,T_{j-1})\right)}\\
&\lesssim |t|^{-(d+4)(\frac{1}{2}-\frac{1}{r})}\|u\|^{\frac{d+4}{d}}_{X(T_{j-1})}\left(\int_{\frac{t}{2}\vee (t-B)}^t|t-s|^{\frac{2}{d+2}-1}ds\right)^\frac{(d+2)(r-2)}{2r}\\
&\lesssim |t|^{-(d+4)(\frac{1}{2}-\frac{1}{r})}\|u\|^{\frac{d+4}{d}}_{X(T_{j-1})}B^\frac{r-2}{r}.
\end{aligned}
\end{equation}
Argued analogously as \eqref{Case-2 P1}, we also have
\begin{equation}\label{Case-2 P2}
\left\|\int_{[\frac{t}{2}\vee (t-B),t]\cap [T_{j-1}, T_j)} e^{i(t-s)\Delta} F(u(s))ds\right\|_{L_x^r}\lesssim |t|^{-(d+4)(\frac{1}{2}-\frac{1}{r})}\|u\|^{\frac{d+4}{d}}_{X(T_j)}B^\frac{r-2}{r}.
\end{equation}

Choose $B=|t|^2\|u\|_{X(T_{j-1})}^{-\frac{4r}{d(r-2)}}$ to minimize the sum of the right hand side of \eqref{fourth P1} and \eqref{Case-2 P1}, which yields
\begin{equation}\label{Case-2-1}
\left\|\int_{[\frac{t}{2},t]\cap [0,T_{j-1})} e^{i(t-s)\Delta} F(u(s))ds\right\|_{L_x^r}\lesssim |t|^{-d(\frac{1}{2}-\frac{1}{r})} C(\|u_0\|_{L_x^2})\|u\|_{X(T_{j-1})}.
\end{equation}
Similarly, choose $B=|t|^2\left(\|u\|_{X(T_j)}\eta^{-1}\right)^{-\frac{4r}{d(r-2)}}$ to minimize the sum of the right hand side of \eqref{fourth P2} and \eqref{Case-2 P2}, which yields
\begin{equation}\label{Case-2-2}
\left\|\int_{[\frac{t}{2},t]\cap [T_{j-1},T_j)} e^{i(t-s)\Delta} F(u(s))ds\right\|_{L_x^r}\lesssim |t|^{-d(\frac{1}{2}-\frac{1}{r})} \eta^\frac{4}{d}\|u\|_{X(T_j)}.
\end{equation}

By \eqref{Case-2-1} and \eqref{Case-2-2}, we have
\begin{equation}\label{Case 2}
\left\|\int_\frac{t}{2}^te^{i(t-s)\Delta}F(u(s))\right\|_{L_x^r}\lesssim |t|^{-d(\frac{1}{2}-\frac{1}{r})} C(\|u_0\|_{L_x^2})\left(\|u\|_{X(T_{j-1})}+\eta^{c(r)}\|u\|_{X(T_j)}\right),
\end{equation}
with $c(r)=\frac{4}{d}$.

\textbf{\underline{Case 3.}} $\frac{2(d+2)}{d}<r<\frac{2d}{d-2}$.
In this case, we take $(p,q)=(\infty,2)$ and we can see $2<\frac{qr}{qr-2}\cdot\frac{d+4}{d}=\frac{(d+4)r}{d(r-1)}\leq r$ and $\frac{pr}{pr-2}=1$. Then we have
\begin{equation}\label{Holder 3}
\|u(s)\|_{L_x^\frac{(d+4)r}{d(r-1)}}\leq \|u(s)\|^{1-\theta}_{L_x^2}\|u(s)\|^\theta_{L_x^r},
\end{equation}
where
\begin{equation*}
\theta=\frac{2d-(d-4)r}{(d+4)(r-2)} \text{ and } 1-\theta=\frac{2dr-4(d+2)}{(d+4)(r-2)}.
\end{equation*}
For the third term, it follows from \eqref{X}, \eqref{Global}, \eqref{Third term} and \eqref{Holder 3} that
\begin{equation}\label{case-3 p1}
\begin{aligned}
&\quad\left\|\int_{[\frac{t}{2}\wedge A, \frac{t}{2}]\cap [0,T_{j-1})} e^{i(t-s)\Delta} F(u(s))ds\right\|_{L_x^r}\\
&\lesssim |t|^{-d(\frac{1}{2}-\frac{1}{r})}\left\|\|u(s)\|_{L_x^\frac{(d+4)r}{d(r-1)}}\right\|^\frac{d+4}{d}_{L^{\frac{d+4}{d}}_s\left([\frac{t}{2}\wedge A, \frac{t}{2}]\cap [0,T_{j-1})\right)}\\
&\leq |t|^{-d(\frac{1}{2}-\frac{1}{r})}\left\|\|u(s)\|^{1-\theta}_{L_x^2}\|u(s)\|^\theta_{L_x^r}\right\|^\frac{d+4}{d}_{L^{\frac{d+4}{d}}_s\left([\frac{t}{2}\wedge A, \frac{t}{2}]\cap [0,T_{j-1})\right)}\\
&\lesssim |t|^{-d(\frac{1}{2}-\frac{1}{r})}\|u\|^{\frac{d+4}{d}\theta}_{X(T_{j-1})}\left\|\|u(s)\|^{1-\theta}_{L_x^2}|s|^{-d(\frac{1}{2}-\frac{1}{r})\theta}\right\|^\frac{d+4}{d}_{L^{\frac{d+4}{d}}_s\left([\frac{t}{2}\wedge A, \frac{t}{2}]\cap [0,T_{j-1})\right)}\\
&\lesssim |t|^{-d(\frac{1}{2}-\frac{1}{r})}\|u\|^{\frac{d+4}{d}\theta}_{X(T_{j-1})}\|u\|^{\frac{d+4}{d}(1-\theta)}_{L_s^\infty L_x^2\left(\left([\frac{t}{2}\wedge A, \frac{t}{2}]\cap [0,T_{j-1})\right)\times\mathbb{R}^d\right)}\left\||s|^{-d(\frac{1}{2}-\frac{1}{r})\theta}\right\|^\frac{d+4}{d}_{L^{\frac{d+4}{d}}_s\left([\frac{t}{2}\wedge A, \frac{t}{2}]\cap [0,T_{j-1})\right)}\\
&\lesssim |t|^{-d(\frac{1}{2}-\frac{1}{r})}\|u\|^{\frac{2d-(d-4)r}{d(r-2)} }_{X(T_{j-1})}\|u\|^{\frac{2dr-4(d+2)}{d(r-2)}}_{L_s^\infty L_x^2}\int_{\frac{t}{2}\wedge A}^\frac{t}{2}|s|^{-\frac{2d-(d-2)r}{2r}-1} ds\\
&\lesssim |t|^{-d(\frac{1}{2}-\frac{1}{r})}\|u\|^{\frac{2d-(d-4)r}{d(r-2)} }_{X(T_{j-1})}\|u\|^{\frac{2dr-4(d+2)}{d(r-2)}}_{L_s^\infty L_x^2}A^{-\frac{2d-(d-2)r}{2r}}\\
&\lesssim |t|^{-d(\frac{1}{2}-\frac{1}{r})}\|u\|^{\frac{2d-(d-4)r}{d(r-2)} }_{X(T_{j-1})}A^{-\frac{2d-(d-2)r}{2r}}C(\|u_0\|_{L^2_x}).\\
\end{aligned}
\end{equation}
Argued analogously as \eqref{case-3 p1}, we also have
\begin{equation}\label{case-3 p2}
\begin{aligned}
\left\|\int_{[\frac{t}{2}\wedge A, \frac{t}{2}]\cap [T_{j-1},T_j)} e^{i(t-s)\Delta} F(u(s))ds\right\|_{L_x^r}&\lesssim |t|^{-d(\frac{1}{2}-\frac{1}{r})}\|u\|^{\frac{2d-(d-4)r}{d(r-2)} }_{X(T_j)}\|u\|^{\frac{2dr-4(d+2)}{d(r-2)}}_{L_s^\infty L_x^2}A^{-\frac{2d-(d-2)r}{2r}}\\
&\lesssim |t|^{-d(\frac{1}{2}-\frac{1}{r})}\|u\|^{\frac{2d-(d-4)r}{d(r-2)} }_{X(T_j)}A^{-\frac{2d-(d-2)r}{2r}}C(\|u_0\|_{L^2_x}).\\
\end{aligned}
\end{equation}

Choose $A=\|u\|_{X(T_{j-1})}^\frac{4r}{d(r-2)}$ to minimize the sum of the right hand side of  \eqref{second P1} and \eqref{case-3 p1}, which yields
\begin{equation}\label{case-3-1}
\left\|\int_{[0,\frac{t}{2}]\cap [0,T_{j-1})}e^{i(t-s)\Delta}F(u(s))\right\|_{L_x^r} \lesssim |t|^{-d(\frac{1}{2}-\frac{1}{r})} C(\|u_0\|_{L_x^2})\|u\|_{X(T_{j-1})}.
\end{equation}
Similarly, choose $A=\|u\|_{X(T_j)}^\frac{4r}{d(r-2)}\eta^{-\frac{4(d+4)r}{d[2d-(d-4)r]}}$ to minimize the sum of the right hand side of \eqref{second P2} and \eqref{case-3 p2}, which yields
\begin{equation}\label{case-3-2}
\left\|\int_{[0,\frac{t}{2}]\cap [T_{j-1},T_j)}e^{i(t-s)\Delta}F(u(s))\right\|_{L_x^r} \lesssim |t|^{-d(\frac{1}{2}-\frac{1}{r})} C(\|u_0\|_{L_x^2}) \eta^\frac{2(d+4)[2d-(d-2)r]}{d[2d-(d-4)r]}\|u\|_{X(T_j)}.
\end{equation}
By \eqref{case-3-1} and \eqref{case-3-2}, it implies
\begin{equation}\label{case 3}
\left\|\int_0^\frac{t}{2}e^{i(t-s)\Delta}F(u(s))\right\|_{L_x^r}\lesssim |t|^{-d(\frac{1}{2}-\frac{1}{r})} C(\|u_0\|_{L_x^2})\left(\|u\|_{X(T_{j-1})}+\eta^{c(r)}\|u\|_{X(T_j)}\right),
\end{equation}
with $c(r)=\frac{2(d+4)[2d-(d-2)r]}{d[2d-(d-4)r]}$.

For the last term, it follows from \eqref{X}, \eqref{Last term} and \eqref{Holder 3} that
\begin{equation}\label{Case-3 P1}
\begin{aligned}
&\quad\left\|\int_{[\frac{t}{2}\vee (t-B),t]\cap [0,T_{j-1})} e^{i(t-s)\Delta} F(u(s))ds\right\|_{L_x^r}\\
&\lesssim \left\||t-s|^{-d(\frac{1}{2}-\frac{1}{r})}\|u(s)\|^\frac{d+4}{d}_{L_x^\frac{(d+4)r}{d(r-1)}}\right\|_{L^1_s\left([\frac{t}{2}\vee (t-B),t]\cap [0,T_{j-1})\right)}\\
&\leq \left\||t-s|^{-d(\frac{1}{2}-\frac{1}{r})}\|u(s)\|^{\frac{d+4}{d}(1-\theta)}_{L_x^2}\|u(s)\|^{\frac{d+4}{d}\theta}_{L_x^r}\right\|_{L^1_s\left([\frac{t}{2}\vee (t-B),t]\cap [0,T_{j-1})\right)}\\
&\leq \left\||t-s|^{-d(\frac{1}{2}-\frac{1}{r})}\|u(s)\|^{\frac{d+4}{d}(1-\theta)}_{L_x^2}\|u\|^{\frac{d+4}{d}\theta}_{X(T_{j-1})}|s|^{-(d+4)(\frac{1}{2}-\frac{1}{r})\theta}\right\|_{L^1_s\left([\frac{t}{2}\vee (t-B),t]\cap [0,T_{j-1})\right)}\\
&\leq |t|^{-(d+4)(\frac{1}{2}-\frac{1}{r})\theta}\|u\|^{\frac{d+4}{d}\theta}_{X(T_{j-1})}\|u\|^{\frac{d+4}{d}(1-\theta)}_{L_s^\infty L_x^2\left([0,T_{j-1})\times\mathbb{R}^d\right)}\int_{\frac{t}{2}\vee (t-B)}^t|t-s|^{\frac{2d-(d-2)r}{2r}-1}ds\\
&\lesssim |t|^{-\frac{2d-(d-4)r}{2r}}\|u\|^{\frac{2d-(d-4)r}{d(r-2)} }_{X(T_{j-1})}\|u\|^{\frac{2dr-4(d+2)}{d(r-2)}}_{L_s^\infty L_x^2}B^\frac{2d-(d-2)r}{2r}\\
&\lesssim |t|^{-\frac{2d-(d-4)r}{2r}}\|u\|^{\frac{2d-(d-4)r}{d(r-2)} }_{X(T_{j-1})}B^\frac{2d-(d-2)r}{2r}C(\|u_0\|_{L^2_x}).
\end{aligned}
\end{equation}
Argued analogously as \eqref{Case-3 P1}, we also have
\begin{equation}\label{Case-3 P2}
\begin{aligned}
\left\|\int_{[\frac{t}{2}\vee (t-B),t]\cap [T_{j-1}, T_j)} e^{i(t-s)\Delta} F(u(s))ds\right\|_{L_x^r}
&\lesssim |t|^{-\frac{2d-(d-4)r}{2r}}\|u\|^{\frac{2d-(d-4)r}{d(r-2)} }_{X(T_j)}\|u\|^{\frac{2dr-4(d+2)}{d(r-2)}}_{L_s^\infty L_x^2}B^\frac{2d-(d-2)r}{2r}\\
&\lesssim |t|^{-\frac{2d-(d-4)r}{2r}}\|u\|^{\frac{2d-(d-4)r}{d(r-2)} }_{X(T_j)}B^\frac{2d-(d-2)r}{2r}C(\|u_0\|_{L^2_x}).
\end{aligned}
\end{equation}

Choose $B=|t|^2\left(\|u\|_{X(T_j)}\eta^{-1}\right)^{-\frac{4r}{d(r-2)}}$ to minimize the sum of the right hand side of \eqref{fourth P1} and \eqref{Case-3 P1}, which yields
\begin{equation}\label{Case-3-1}
\left\|\int_{[\frac{t}{2},t]\cap [0,T_{j-1})} e^{i(t-s)\Delta} F(u(s))ds\right\|_{L_x^r} \lesssim |t|^{-d(\frac{1}{2}-\frac{1}{r})} C(\|u_0\|_{L_x^2})\|u\|_{X(T_{j-1})}.
\end{equation}
Similarly, choose $B=|t|^2\|u\|_{X(T_j)}^{-\frac{4r}{d(r-2)}}\eta^{\frac{4(d+4)r}{d[2d-(d-4)r]}}$ to minimize the sum of the right hand side of \eqref{fourth P2} and \eqref{Case-3 P2}, which yields
\begin{equation}\label{Case-3-2}
\left\|\int_{[\frac{t}{2},t]\cap [T_{j-1},T_j)} e^{i(t-s)\Delta} F(u(s))ds\right\|_{L_x^r}\lesssim |t|^{-d(\frac{1}{2}-\frac{1}{r})} C(\|u_0\|_{L_x^2}) \eta^\frac{2(d+4)[2d-(d-2)r]}{d[2d-(d-4)r]}\|u\|_{X(T_j)}.
\end{equation}

It follows from \eqref{Case-3-1} and \eqref{Case-3-2} that
\begin{equation}\label{Case 3}
\left\|\int_\frac{t}{2}^te^{i(t-s)\Delta}F(u(s))\right\|_{L_x^r}\lesssim |t|^{-d(\frac{1}{2}-\frac{1}{r})} C(\|u_0\|_{L_x^2})\left(\|u\|_{X(T_{j-1})}+\eta^{c(r)}\|u\|_{X(T_j)}\right),
\end{equation}
with $c(r)=\frac{2(d+4)[2d-(d-2)r]}{d[2d-(d-4)r]}$.
\vskip 0.5cm
Combining \eqref{case 1}, \eqref{case 2}, \eqref{case 3}, \eqref{Case 1}, \eqref{Case 2} and \eqref{Case 3}, we arrive at
\begin{equation*}
    |t|^{d(\frac{1}{2}-\frac{1}{r})}\|u(t)\|_{L_x^r}\lesssim \|u_0\|_{L_x^{r'}}+C(\|u_0\|_{L_x^2})\left(\|u\|_{X(T_{j-1})}+\eta^{c(r)}\|u\|_{X(T_j)}\right),
\end{equation*}
uniformly for $t\in [0,T_j)$, which proves \eqref{bootstrap} and we complete the proof.

\end{document}